\documentclass[12pt, reqno]{amsart}
\usepackage{amsfonts}
\usepackage{graphicx}
\usepackage{subfigure}
\usepackage{latexsym}
\usepackage{amsmath}
\usepackage{amssymb}
\usepackage{color}
\usepackage{verbatim}
 \newtheorem{theorem}{Theorem}[section]
 \newtheorem{Def}[theorem]{Definition}
 \newtheorem{Prop}[theorem]{Proposition}
 \newtheorem{Lem}[theorem]{Lemma}
 \newtheorem{Cor}[theorem]{Corollary}
 
  \newtheorem*{Rem*}{Remark}

 \numberwithin{equation}{section}
 
\begin{document}
\title{On the classification of fractal squares}

\author{Jun Jason Luo}\address{College of Mathematics and Statistics, Chongqing University, Chongqing 401331, China}\email{jun.luo@cqu.edu.cn}

\author{Jing-Cheng Liu}\address{Key Laboratory of High Performance Computing and Stochastic Information Processing (Ministry of Education of China), College of Mathematics and Computer Science, Hunan Normal University, Changsha 410081, Hunan,  P.R. China} \email{liujingcheng11@126.com}

\thanks{The research is supported in part by the NNSF of China (No.11301175, No.11301322, No.11571104), the program for excellent talents in Hunan Normal University (No. ET14101), Specialized Research Fund for the Doctoral Program of Higher Education of China (20134402120007), Foundation for Distinguished Young Talents in Higher Education of Guangdong (2013LYM\_0028).}


\date{\today}

\begin{abstract}
In \cite{LaLuRa13}, the authors completely classified the topological structure of so called {\it fractal square} $F$ defined by $F=(F+{\mathcal D})/n$, where ${\mathcal{D}}\subsetneq\{0,1,\dots,n-1\}^2, n\ge 2$. In this paper, we further provide simple criteria for the $F$ to be totally disconnected, then we discuss the Lipschitz classification of $F$ in the case $n=3$, which is an attempt to consider non-totally disconnected sets.
\end{abstract}

\maketitle

\begin{section}{Introduction}

For $n\geq 2$, let ${\mathcal{D}}=\{d_1,\dots, d_m\}\subsetneq \{0,1,\dots,n-1\}^2$ be a digit set with cardinality $\#{\mathcal{D}}=m$, and let $\{S_i\}_{i=1}^m$  be an iterated function system (IFS) on ${\mathbb R}^2$, where $S_i(x)=\frac{1}{n}(x+d_i)$ where $d_i\in {\mathcal D}$. Then there exists a unique self-similar set $F\subset {\mathbb R}^2$  \cite{Fa90} satisfying the set equation:
\begin{equation}\label{set equation of fractal square}
F= \bigcup_{i=1}^m S_i(F)=\frac{1}{n}(F+{\mathcal D})
\end{equation} which is called a {\it fractal square} \cite{LaLuRa13}. The geometric construction of a fractal square seems like that of middle third Cantor set:  First we divide a unit square into $n^2$ small equal squares of which $m$ small squares are kept and the rest discarded, the positions of the $m$ chosen squares depend on  ${\mathcal D}$; Secondly, repeat the first step on every chosen square and continue in this way, we then obtain a fractal square by taking limits. Obviously, for fixed $n$, the $F$ only relies on the digit set ${\mathcal D}$.

In \cite{LaLuRa13}, Lau, Rao and one of the authors gave a detailed study on the topological structure of $F$, they completely classified the topology of $F$ by three types: (I) $F$ is totally disconnected; (II) $F$ contains a non-trivial component  which is not a line segment; and (III) All non-trivial components of $F$ are parallel line segments.

Let  ${\mathcal{F}}_{n,m}$ denote the collection of all fractal squares satisfying (\ref{set equation of fractal square}). It is easy to see that the fractal squares in ${\mathcal{F}}_{n,m}$ have the common Hausdorff dimension ($\log m/\log n$) but  distinct topological structures. In the above three types,  the fractal squares of type (I) are called Cantor-type sets which play an important role in fractal geometry and dynamical systems, so we will give a further study on this case. Especially, we provide simple criteria for the existence of type (I) in ${\mathcal{F}}_{n,m}$.

Two sets $E$ and $F$ on ${\mathbb R}^d$ are said to be {\it Lipschitz equivalent}, and denoted by $E\simeq F$, if there is a bi-Lipschitz map  $g$ from $E$ onto $F$, i.e., $g$ is a bijection and  there is a constant $C>0$ such that
$$
C^{-1}|x-y|\leq |g(x)-g(y)|\leq C|x-y|,\quad   \forall  \ x,y\in E.
$$
It is well-known that if $E \simeq F$ then they have the same Hausdorff dimension, but the converse is never true. Lipschitz classification of sets has attracted a lot of interests in the literature. In fractal geometry, the fundamental works were due to Cooper and Pignartaro \cite{CoPi88} and  Falconer and Marsh  \cite{FaMa92} on  Cantor sets. Recently, many generalizations  on totally disconnected self-similar sets (Cantor-type sets) have been extensively studied (see \cite{DeLaLu14},\cite{LuLa13},\cite{RaRuXi06},\cite{RaRuWa10},\cite{Ro10},\cite{RuWaXi12},\cite{XiXi10},\cite{XiXi13}). But there are few  results on non-totally disconnected cases \cite{WeZhDe12}.  Motivated by that, our aim of the paper is to make an attempt in this direction.

For ${\mathcal{F}}_{n,m}$, the Lipschitz equivalence class is denoted by ${\mathcal{F}}_{n,m}/{\simeq}$.  When $n=3, m=2,3,4,5$, we have

\begin{theorem}\label{thm1.2}
$\#({\mathcal{F}}_{3,2}/{\simeq}) =1; \ \#({\mathcal{F}}_{3,3}/{\simeq}) =\#({\mathcal{F}}_{3,4}/{\simeq}) =2$; and $\#({\mathcal{F}}_{3,5}/{\simeq}) \leq 10$.
\end{theorem}

The first three classes are simple, while ${\mathcal F}_{3,5}$ is complicated, as it contains all the three types of fractal squares. The complete classification seems very difficult, but we conjecture that $\#({\mathcal F}_{3,5}/\simeq)=10$ (see remarks in Section 4).

The paper is organized as follows: In Section 2, we discuss several criteria for a fractal square to be totally disconnected.  We prove Theorem \ref{thm1.2} by using various methods (see Theorems \ref{th3.4}, \ref{th3.3}, \ref{th3.7}, and \ref{th3.11}) in Section 3, and give some remarks on other cases in Section 4. Finally, we include all figures of fractal squares in ${\mathcal{F}}_{3,5}, {\mathcal{F}}_{3,6}, {\mathcal{F}}_{3,7}$ and ${\mathcal{F}}_{3,8}$ in an appendix.
\end{section}

\bigskip

\begin{section}{Criteria for total disconnectedness}

For fractal square $F$ as in (\ref{set equation of fractal square}), we define a set on ${\mathcal D}$ by  $${\mathcal{E}}=\{(d_i,d_j):(F+d_i)\cap(F+d_j)\ne\emptyset, d_i,d_j\in {\mathcal{D}}\}.$$
We say that $d_i, d_j$ are {\it $\mathcal{E}$-connected} if there exists a finite sequence $\{d_{j_1},\dots,d_{j_k}\}\subset {\mathcal{D}} $ such that $d_i=d_{j_1},d_j=d_{j_k}$ and $(d_{j_l},d_{j_{l+1}})\in {\mathcal{E}}, 1\leq l \leq k-1$.  The following criterion for connectedness  was first proved by \cite{Ha85} and rediscovered by \cite{KiLa00}.

\begin{Lem}\label{th2.1}
A fractal square $F$ with a digit set $\mathcal{D}$ is connected if and only if  any two $d_i, d_j\in {\mathcal{D}}$  are $\mathcal{E}$-connected.
\end{Lem}

Let $B=[0,1]^2$ be the unit square, $\Sigma=\{1,\dots,m\}$. Let ${\mathcal{D}}_1={\mathcal{D}}$ and ${\mathcal{D}}_{k+1}={\mathcal{D}}+n{\mathcal{D}}_{k}$, then
\begin{equation}\label{digitset}
{\mathcal{D}}_k=\{d_{\bf u} :=d_{j_k}+nd_{j_{k-1}}+ \cdots + n^{k-1}d_{j_1}: {\bf u}=j_1\cdots j_k\in{\Sigma}^k\}, \quad k\geq 1.
\end{equation}
Denote by $S_{\bf u}(B)=S_{j_1}\circ\cdots\circ S_{j_k}(B)=n^{-k}(B+d_{\bf u})$, we call such $S_{\bf u}(B)$ (or any translation of $n^{-k}$ scaling of $B$) a {\it $k$-square}. Obviously, we have
\begin{equation}\label{eq2.1}
F=\bigcap_{k=1}^{\infty}\bigcup_{{\bf u}\in {\Sigma^k}}S_{\bf u}(B).
\end{equation}

By letting $F^{(k)}=\bigcup_{{\bf u}\in {\Sigma^k}}S_{\bf u}(B)$, we call $F^{(k)}$ a {\it $k$-th approximation} of the fractal square $F$.

\begin{Cor}\label{connectedness lemma}
A fractal square $F$ is connected if and only if the $k$-th approximation  $F^{(k)}$  is connected  for any $k\geq 1$.
\end{Cor}

\begin{Def}\label{lambda-path}
In $B$, a vertical path is a curve starting at point $(x,0)$ and ending at point $(x,1)$ for some $x\in [0,1]$; a horizontal path  is a curve   starting at point $(0,y)$ and ending at point  $(1,y)$ for some $y\in [0,1]$; a cross path  is the union of one vertical path and one horizontal path;  a $\lambda$-path  is the union $\gamma_1\cup\gamma_2\cup\gamma_3$ where $\gamma_i$ are three arcs connecting  an interior point of $B$ and three corners of $B$, respectively. (see Figure \ref{fig.fourpath})
\end{Def}

Obviously, a vertical path and a horizontal path meet each other, so a cross path is connected and reaches four points of the four sides of $B$, respectively.  A $\lambda$-path is also connected. Intuitively, the shape of the $\lambda$-path looks like the letter `$\lambda$' or its rotations. The simplest $\lambda$-path  may be the union of a diagonal and half of the other in $B$. From (\ref{eq2.1}),  it can be seen that  $B\setminus F$ contains a vertical path  if and only if there exist an integer $k\ge 1$ and  a chain of adjacent $k$-squares outside  $F^{(k)}$ which begins with $[\frac{j}{n^k},\frac{j+1}{n^k}]\times[0,\frac{1}{n^k}]$ and ends with $[\frac{j}{n^k},\frac{j+1}{n^k}]\times[1-\frac{1}{n^k}, 1]$ for some $j\in \{0,1, \dots, n^k-1\}$. Similarly for the cross path and the $\lambda$-path.  (see Figure \ref{fig.pathsquare})

\begin{figure} [h]
\centering
  \includegraphics[height=4.5cm]{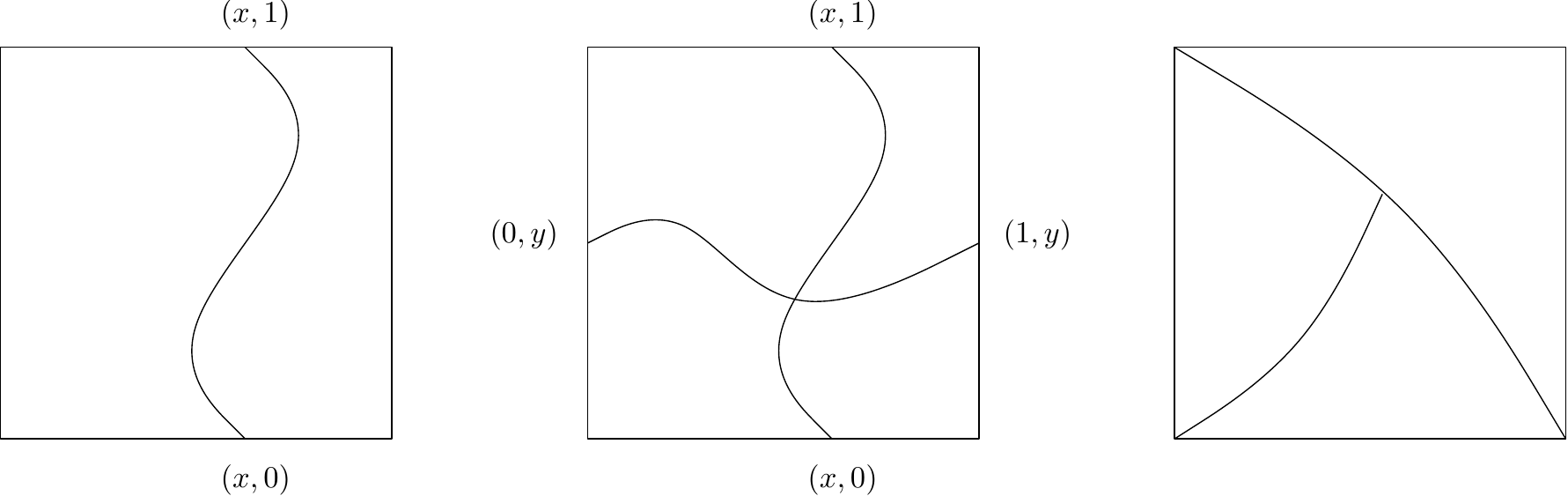}
\caption{From left to right: a vertical path, a cross path and a $\lambda$-path}\label{fig.fourpath}
\end{figure}

\begin{figure} [h]
\centering
  \includegraphics[height=4.5cm]{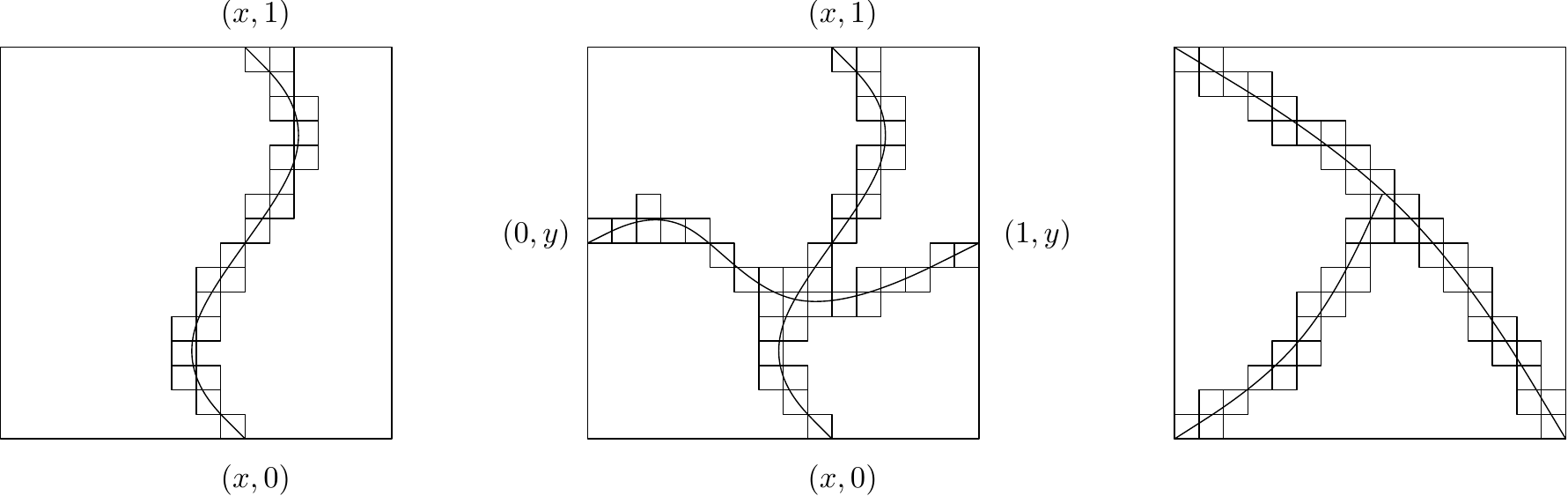}
\caption{Paths covered by squares}\label{fig.pathsquare}
\end{figure}

\bigskip

The main use of the above four paths is to verify the total disconnectedness of $F$.

\begin{Prop}[\cite{Ro10}] \label{total disco thm1}
A fractal square $F$ is totally disconnected if and only if $B\setminus F$ has a cross path.
\end{Prop}

The following criterion is more convenient for many cases in our consideration. Note that $F$ contains a vertical (horizontal) line segment if and only if $F^{(1)}$ does.

\begin{theorem}\label{total disconnected thm}
A fractal square $F$ is totally disconnected if and only if $F$  contains no vertical line segments and $B\setminus F$ contains a vertical path.
\end{theorem}

\begin{proof}
If $F$ is totally disconnected, then the necessity is obvious since $B\setminus F$ is open and pathwise connected. For the converse part,   let $C$ be a component of $F$, and $\text{Proj}_x C$  denote  the orthogonal projection of $C$ on the $x$-axis, then $\text{Proj}_x C$ is also pathwise connected.  We claim $|\text{Proj}_x C|=0$. Indeed, if otherwise, $|\text{Proj}_x C|>0$. Choose an integer $k$ large enough such that
\begin{align*}
& C\cap \left[\frac{i}{n^{k}},\frac{i+1}{n^{k}}\right]\times[0,1]  \ne\emptyset, \\
& C\cap \left[\frac{i+1}{n^{k}},\frac{i+2}{n^{k}}\right]\times[0,1]  \ne\emptyset, \\
& C\cap \left[\frac{i+2}{n^{k}},\frac{i+3}{n^{k}}\right]\times[0,1]  \ne\emptyset
\end{align*}
hold for some $i\in\{0, 1, \dots, n^{k}-3\}$.  Let $I_j=[\frac{i+1}{n^{k}},\frac{i+2}{n^{k}}]\times[\frac{j}{n^{k}},\frac{j+1}{n^{k}}], j=0,1,\dots, n^{k}-1$ be the $k$-squares in the rectangle  $[\frac{i+1}{n^{k}},\frac{i+2}{n^{k}}]\times[0,1]$. Suppose $\alpha$ is a vertical path of $B\setminus F$. If $I_j$ belongs to the $k$-th approximation of $F$, we denote it by $S_{\bf u}(B)$ for some ${\bf u}\in \Sigma^{k}$.  Then $S_{\bf u}(B\setminus F)$  contains a  path $S_{\bf u}(\alpha):=\alpha_j$; if not, then $I_j\subset B\setminus F$. We can take a vertical line $\beta_j$ in $I_j$ with the same horizontal coordinate as the $\alpha_j$. Hence we construct a vertical path in $B\setminus F$ by joining the paths $\alpha_j, \beta_j$, which  separate the component $C$. Thus, the $C$ must lie in one vertical line. By the assumption, $C$ can not be a vertical line segment, which implies $C$ is just a singleton. Therefore, $F$ is totally disconnected.
\end{proof}

\begin{Prop} \label{total disco thm2}
Let $F$ be a fractal square. If $B\setminus F$ contains a $\lambda$-path, then $F$ is totally disconnected. Conversely, if $F$ is totally disconnected and at most one corner of $B$ is in $F$, then there exists a $\lambda$-path in $B\setminus F$.
\end{Prop}

\begin{proof}
The proof is essentially the same as above. We mention that if $F$ is totally disconnected and at most one corner of $B$ is in $F$, then $B\setminus F$ contains at least three corners of $B$. Hence we can construct a $\lambda$-path in $B\setminus F$ by using the  pathwise connectedness of $B\setminus F$.
\end{proof}

\begin{theorem}\label{thm.totally.dis.exist.}
If $m \leq n^2-n-[\frac{n}{2}]$ then ${\mathcal{F}}_{n,m}$ contains a totally disconnected fractal square.
\end{theorem}

\begin{proof}
Let ${\mathcal{D}}_1=\{(i,i): i=0,1,\dots,[\frac{n}{2}]\}\cup\{(j,n-j-1): j=0,1,\dots, n-1\}$, and a digit set ${\mathcal{D}}= \{0,1,\dots,n-1\}^2 \setminus {\mathcal{D}}_1$. Then $\#{\mathcal D} = n^2-n-[\frac{n}{2}]:=m$, and  $F=\frac{1}{n}(F+{\mathcal{D}})$ belongs to ${\mathcal F}_{n,m}$. Since the set ${\mathcal D}_1$ determines a $\lambda$-path in $B\setminus F^{(2)}$, so in $B\setminus F$, it implies that $F$ is totally disconnected by Proposition \ref{total disco thm2}.
\end{proof}

We believe the converse is also true, but we can not find a valid proof yet.
\end{section}

\bigskip

\begin{section}{Classification of fractal squares when $n=3$}

\begin{Lem}[\cite{LuLa13},\cite{XiXi10}]\label{lip on totaldis}
Let $F, F'\in {\mathcal{F}}_{n,m}$ be two fractal squares.  If $F, F'$ are totally disconnected then $F\simeq F'$.
\end{Lem}

However, if two fractal squares are not totally disconnected, there are few results about their Lipschitz equivalence. In this section, we make an attempt on some special cases, such as connected fractal squares or fractal squares containing parallel line segments.  We try to classify the Lipschitz equivalence classes of ${\mathcal{F}}_{n,m}$  for $n=3, m=2,3,4,5$. For convenience, we use an $n\times n$ matrix $M=(m_{ij})_{1\leq i,j\leq n}$ to represent a fractal square $F$  where
\begin{equation*}
m_{ij}= \begin{cases} 1 & \text{if} \  (j-1,n-i)\in {\mathcal{D}} \\
  0  &  \text{otherwise}. \end{cases}
\end{equation*}
We call $M$ the \emph{label matrix} of $F$. It is easy to see that there is a one-to-one correspondence between $F$ and $M$. The geometric meaning of $M$ is: if we divide the unit square $B$ into $n^2$ small squares and pick out $m$ small squares (depending on the digit set ${\mathcal D}$) as our first approximation of $F$, then the nonzero entries of $M$ represent the relevant positions of the $m$ chosen small squares while the zero entries represent the relevant positions of the $n^2-m$ unchosen small squares. So we prefer to use the label matrix to depict the fractal square for simplicity.

Geometrically, two sets are called {\it congruent} if one can be transformed into the other by some rigid motions. From (\ref{eq2.1}), it is seen that two fractal squares are congruent if and only if their first approximations are congruent, which can be immediately observed from the label matrices.

\begin{Lem}\label{lip on liner}
Let $g: {\mathbb R}^d\to {\mathbb R}^d$ be a linear transformation defined by  $g(x)=Ax+v$ where $A$ is a $d\times d$ invertible matrix and  $v\in {\mathbb R}^d$. Then $g$ is a bi-Lipschitz map.
\end{Lem}

\begin{proof}
Since $A$ is invertible, for any $x,y\in{\mathbb R}^d$, we have $$|A(x-y)|\le \|A\||x-y|$$ and $$|x-y|=|A^{-1}A(x-y)|\le \|A^{-1}\||A(x-y)|$$ where $\|A\|$ denotes the norm of matrix $A$.
Hence $$\|A^{-1}\|^{-1}|x-y|\le |g(x)-g(y)|\le\|A\||x-y|$$ proving  that $g$ is a bi-Lipschitz map.
\end{proof}

\begin{theorem}\label{th3.4}
$\#({\mathcal{F}}_{3,2}/{\simeq}) =1; \  \#({\mathcal{F}}_{3,3}/{\simeq}) = \#({\mathcal{F}}_{3,4}/{\simeq}) =2$.
\end{theorem}

\begin{proof}
Since  $\log 2/\log 3 <1$,  all the fractal squares in ${\mathcal{F}}_{3,2}$ are totally disconnected \cite{Fa90}. Hence $\#({\mathcal{F}}_{3,2}/{\simeq}) =1$ by Lemma \ref{lip on totaldis}.

In ${\mathcal{F}}_{3,3}$, every fractal square is either totally disconnected or connected (a line segment). Hence  $\#({\mathcal{F}}_{3,3}/{\simeq}) =2$.

By Theorem \ref{thm.totally.dis.exist.}, ${\mathcal{F}}_{3,4}$ contains totally disconnected fractal squares, and they are Lipschitz equivalent by Lemma \ref{lip on totaldis}. Moreover, it can be easily checked that, up to congruence, there are only $6$ different non-totally disconnected fractal squares, denoted by $F_i=\frac13(F_i+{\mathcal D}_i)$ where $i=1,\dots, 6$. The corresponding label  matrices are listed as follows:
{\scriptsize  \begin{eqnarray*}
\left[\begin{array}{ccc}
1 & 0 & 0 \\
0 & 0 & 0 \\
1 & 1 & 1
\end{array}\right], ~
\left[\begin{array}{ccc}
0 & 1 & 0 \\
0 & 0 & 0 \\
1 & 1 & 1
\end{array}\right], ~
\left[\begin{array}{ccc}
0 & 0 & 0 \\
1 & 0 & 0 \\
1 & 1 & 1
\end{array}\right],
\left[\begin{array}{ccc}
0 & 0 & 0 \\
0 & 1 & 0 \\
1 & 1 & 1
\end{array}\right], ~
\left[\begin{array}{ccc}
0 & 0 & 1 \\
0 & 1 & 1 \\
1 & 0 & 0
\end{array}\right], ~
\left[\begin{array}{ccc}
0 & 0 & 1 \\
0 & 1 & 0 \\
1 & 0 & 1
\end{array}\right].
\end{eqnarray*} }

We define a linear  transformation $g: {\mathbb R}^2\to{\mathbb R}^2$ by $g(x)=Ax$ where $A=\left[\begin{array}{cc}
1 & 1/2 \\
0 & 1
\end{array}\right]$. Then ${\mathcal D}_2= A{\mathcal D}_1$.  Hence  $AF_1=\frac13(AF_1+A{\mathcal D}_1)=\frac13(AF_1+{\mathcal D}_2)$, implying $F_2=g(F_1)$ by the uniqueness of attractor. So we get $F_1\simeq F_2$ as $g$ is a bi-Lipschitz map by Lemma \ref{lip on liner}.

Similarly, it is easy to verify  that ${\mathcal D}_3= A_1{\mathcal D}_1, \  {\mathcal D}_4= A_1{\mathcal D}_2, \  {\mathcal D}_6= A_2{\mathcal D}_4$, and ${\mathcal D}_5= A_3{\mathcal D}_4$,  where
$$A_1=\left[\begin{array}{cc}
1 & 0 \\
0 & 1/2
\end{array}\right],\quad A_2=\left[\begin{array}{cc}
1 & 1 \\
1 & -1
\end{array}\right],\quad  A_3=\left[\begin{array}{cc}
1 & 1 \\
1 & 0
\end{array}\right].$$
Therefore, $F_1\simeq F_2\simeq\cdots\simeq F_6$, proving $\#({\mathcal{F}}_{3,4}/{\simeq}) =2$.
\end{proof}

In ${\mathcal{F}}_{3,5}$, the total number of fractal squares is $C_9^5$.  Up to congruence, there are $21$ distinct  fractal squares among them. However, in the rest of this section, we will show that there are at most $10$ Lipschitz equivalence classes. First we use $\{F_i\}_{i=1}^{21}$ to denote $21$ fractal squares, and  each $F_i$ takes the following $M_i$ as its label matrix:

\bigskip

\noindent {\bf Type (I):} totally disconnected fractal squares:

{\scriptsize \begin{eqnarray*}
M_1= \left[\begin{array}{ccc}
0 & 1 & 0 \\
1 & 0 & 1 \\
1 & 0 & 1
\end{array}\right], M_2=\left[\begin{array}{ccc}
1 & 0 & 0 \\
0 & 1 & 1 \\
1 & 1 & 0
\end{array}\right], M_3=\left[\begin{array}{ccc}
1 & 1 & 0 \\
0 & 0 & 1 \\
1 & 0 & 1
\end{array}\right], M_4=\left[\begin{array}{ccc}
1 & 1 & 0 \\
0 & 0 & 1 \\
0 & 1 & 1
\end{array}\right],  M_5= \left[\begin{array}{ccc}
1 & 1 & 0 \\
1 & 0 & 1 \\
0 & 1 & 0
\end{array}\right].\end{eqnarray*} }

\noindent {\bf Type (II):} connected fractal squares:

{\scriptsize \begin{eqnarray*}
&& M_6 = \left[\begin{array}{ccc}
0 & 1 & 0 \\
0 & 1 & 0 \\
1 & 1 & 1
\end{array}\right], M_7=\left[\begin{array}{ccc}
1 & 0 & 0 \\
1 & 0 & 0 \\
1 & 1 & 1
\end{array}\right], M_8=\left[\begin{array}{ccc}
0 & 0 & 1 \\
0 & 1 & 0 \\
1 & 1 & 1
\end{array}\right], M_9=\left[\begin{array}{ccc}
1 & 0 & 1 \\
0 & 1 & 0 \\
1 & 0 & 1
\end{array}\right], \\
&& M_{10} = \left[\begin{array}{ccc}
0 & 1 & 0 \\
1 & 1 & 1 \\
0 & 1 & 0
\end{array}\right], M_{11}=\left[\begin{array}{ccc}
0 & 1 & 1 \\
0 & 1 & 0 \\
1 & 1 & 0
\end{array}\right].
\end{eqnarray*} }

\noindent {\bf Type (III):} fractal squares containing parallel line segments:

{\scriptsize \begin{eqnarray*}
&& M_{12}=\left[\begin{array}{ccc}
1 & 0 & 1 \\
0 & 0 & 0 \\
1 & 1 & 1
\end{array}\right], M_{13}=\left[\begin{array}{ccc}
0 & 0 & 0 \\
1 & 0 & 1 \\
1 & 1 & 1
\end{array}\right],  M_{14}= \left[\begin{array}{ccc}
0 & 0 & 0 \\
1 & 1 & 0 \\
1 & 1 & 1
\end{array}\right],  M_{15}=\left[\begin{array}{ccc}
1 & 1 & 0 \\
0 & 0 & 0 \\
1 & 1 & 1
\end{array}\right], \\
&& M_{16}=\left[\begin{array}{ccc}
0 & 0 & 1 \\
0 & 1 & 1 \\
1 & 1 & 0
\end{array}\right],  M_{17}=\left[\begin{array}{ccc}
0 & 1 & 0 \\
0 & 0 & 1 \\
1 & 1 & 1
\end{array}\right], M_{18}= \left[\begin{array}{ccc}
1 & 0 & 0 \\
0 & 0 & 1 \\
1 & 1 & 1
\end{array}\right], M_{19}=\left[\begin{array}{ccc}
0 & 0 & 1 \\
1 & 1 & 0 \\
1 & 1 & 0
\end{array}\right], \\
&& M_{20}=\left[\begin{array}{ccc}
0 & 1 & 0 \\
0 & 1 & 1 \\
1 & 1 & 0
\end{array}\right], M_{21}=\left[\begin{array}{ccc}
1 & 0 & 1 \\
0 & 1 & 0 \\
1 & 1 & 0
\end{array}\right].
\end{eqnarray*} }

By the criteria in the last section, especially Theorem \ref{total disconnected thm},  fractal squares of type (I) are indeed totally disconnected. Hence $F_i, 1\le i \le 5$ are Lipschitz equivalent by Lemma \ref{lip on totaldis}.  For types (II) and (III), we have

\begin{theorem}\label{th3.3}
$F_7\simeq F_8; \  F_9\simeq F_{10}\simeq F_{11}; \  F_{12}\simeq F_{13}; \ F_{14}\simeq F_{15}\simeq F_{16}; \ F_{19}\simeq F_{20}$.
\end{theorem}

\begin{proof}
Let $F_i=\frac13(F_i+{\mathcal D}_i), i=1,\dots, 21$, and let
\begin{eqnarray*}
& & A_1=\left[\begin{array}{cc}
1 & 1 \\
0 & 1
\end{array}\right], \quad  A_2=\left[\begin{array}{cc}
1/2 & 1/2 \\
0 & 1
\end{array}\right], \quad  A_3=\left[\begin{array}{cc}
1 & 0 \\
0 & 1/2
\end{array}\right], \\
& & A_4=\left[\begin{array}{cc}
1 & 1 \\
1 & 0
\end{array}\right],  \quad  A_5=\left[\begin{array}{cc}
1 & 1 \\
-1 & 1
\end{array}\right], \quad  A_6=\left[\begin{array}{cc}
1 & -1 \\
1 & 0
\end{array}\right].
\end{eqnarray*}
Then  ${\mathcal D}_8= A_1{\mathcal D}_7,  {\mathcal D}_{11}= A_2{\mathcal D}_9, {\mathcal D}_{13}=A_3{\mathcal D}_{12},  A_3{\mathcal D}_{15}= {\mathcal D}_{14}= A_4^{-1}{\mathcal D}_{16}$. By defining linear  transformations $g_i(x)=A_ix$ for $i=1,2,3,4$, we can obtain $g_1(F_7)=F_8, g_2(F_9)=F_{11}, g_3(F_{15})=F_{14}$ and $g_4(F_{14})=F_{16}$, where  $g_i$ are bi-Lipschitz maps.

Moreover, let $g_5(x)=A_5x+v, \  g_6(x)=A_6x+\frac{v'}{2}$  where $v=\frac12\left[\begin{array}{c}
-1\\ 1\end{array}\right], \ v'=\frac12\left[\begin{array}{c}
1\\ 0\end{array}\right]$.  Then ${\mathcal D}_9=A_5{\mathcal D}_{10}+2v$ and ${\mathcal D}_{20}=A_6{\mathcal D}_{19}+v'$. Hence $$g_5(F_{10})=g_5(\frac13(F_{10}+{\mathcal D}_{10}))=\frac 13 (A_5F_{10}+A_5{\mathcal D}_{10}+3v)=\frac13(g_5(F_{10})+{\mathcal D}_9)$$ and $$g_6(F_{19})=g_6(\frac13(F_{19}+{\mathcal D}_{19}))=\frac 13 (A_6F_{19}+A_6{\mathcal D}_{19}+\frac32v')=\frac13(g_6(F_{19})+{\mathcal D}_{20}).$$
That implies  $g_5(F_{10})=F_9$ and $g_6(F_{19})=F_{20}$, finishing the proof.
\end{proof}

\bigskip

\begin{figure}[h]
 \centering
 \subfigure[]{
 \includegraphics[width=4cm]{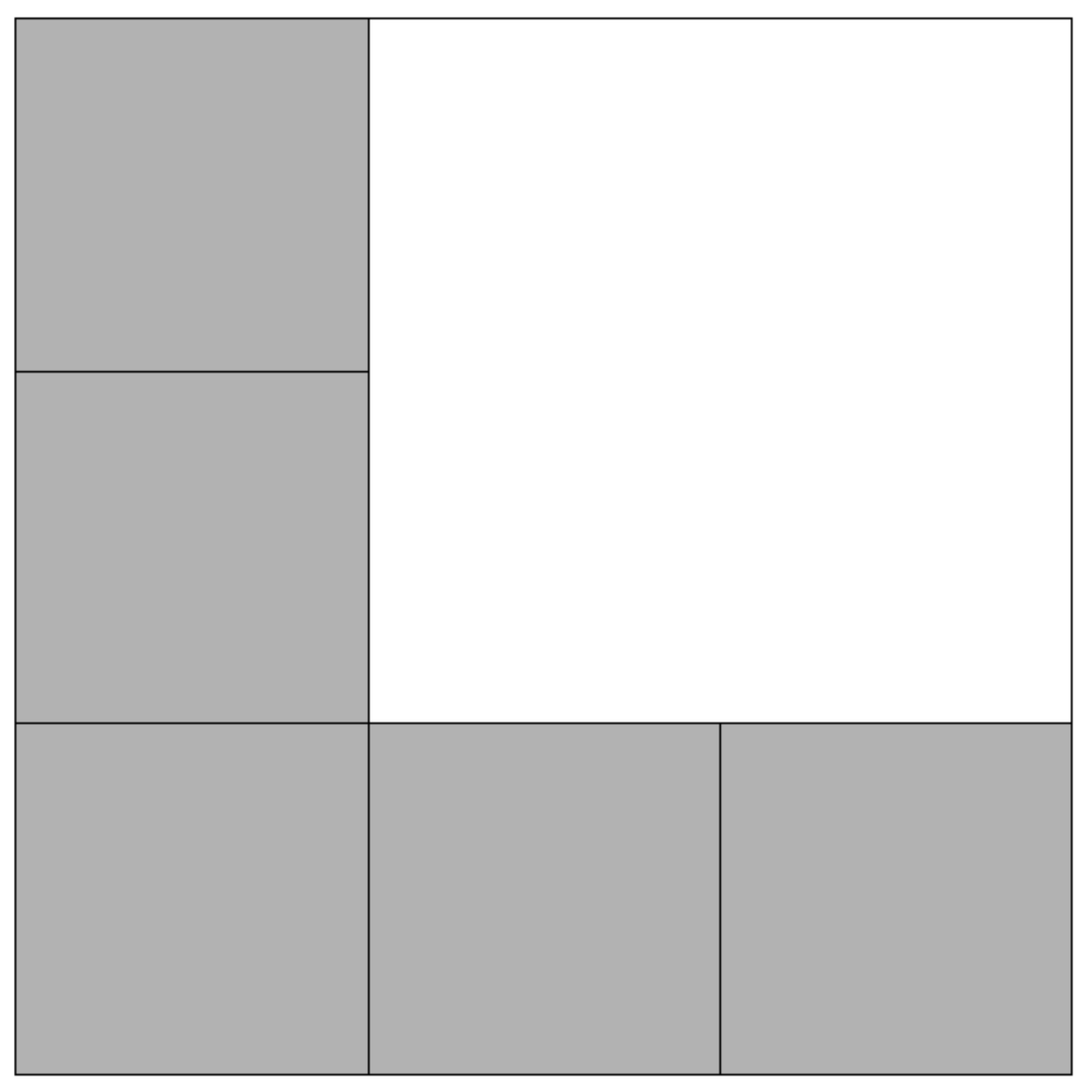}
}
\qquad
\subfigure[]{
 \includegraphics[width=4cm]{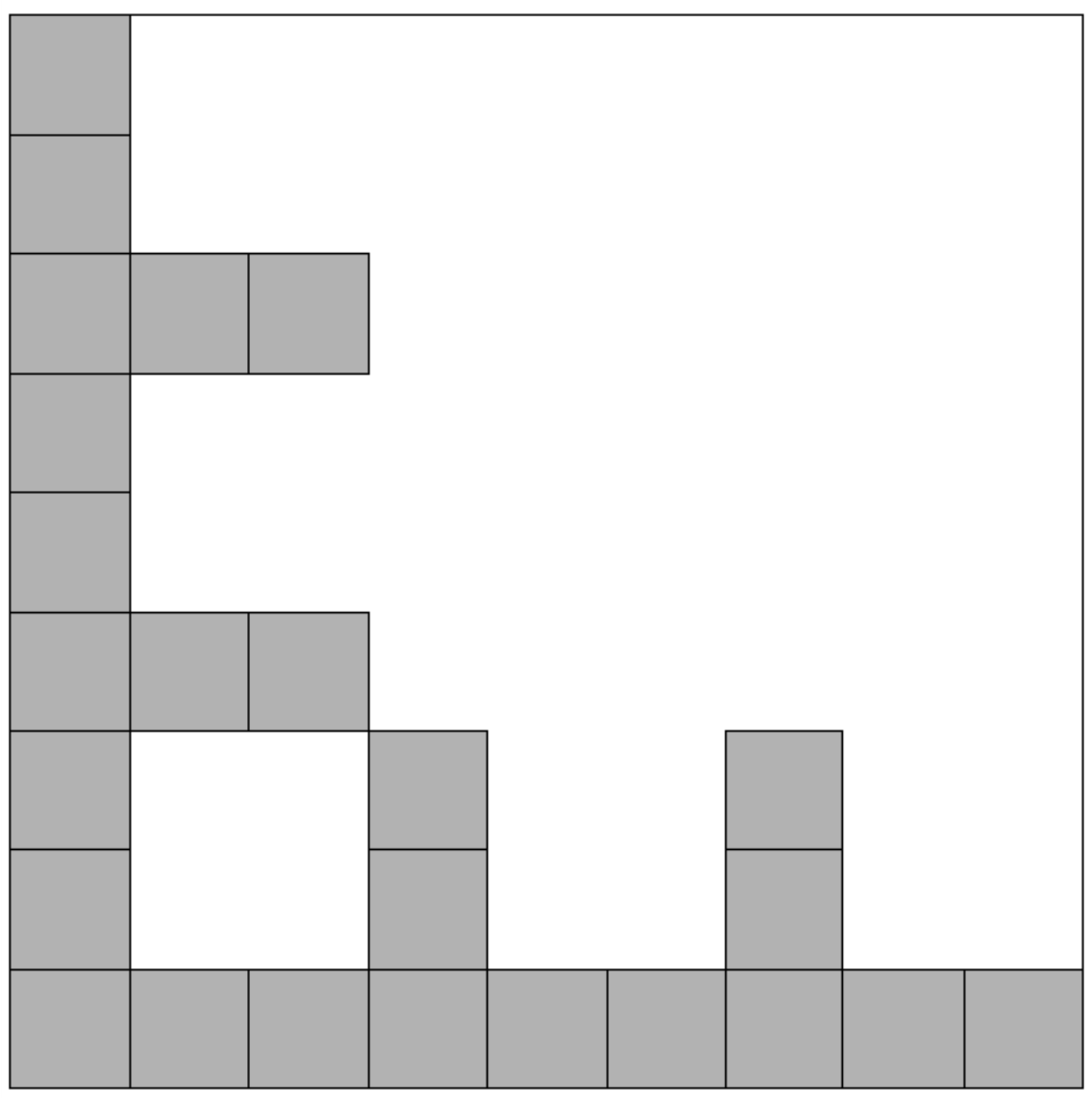}
}
\qquad
\subfigure[]{
 \includegraphics[width=4cm]{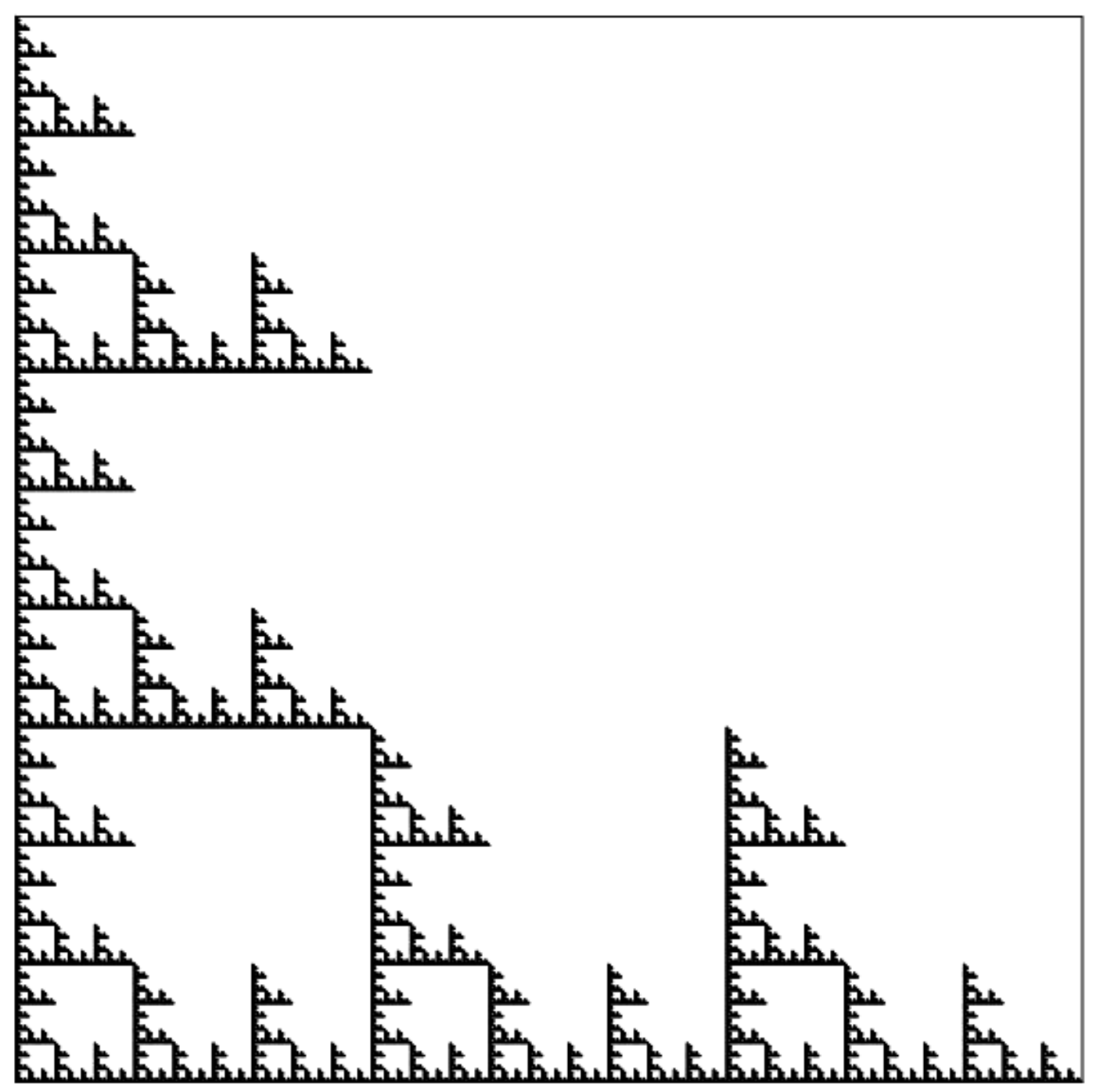}
}
\caption{Fractal square $F_7$}\label{fig-lem3.6}
\end{figure}

\begin{Lem} {\label{countable square lemma}}
$F_7$ is a connected set which equals the closure of a union of infinitely countable circles, and so does $F_8$.
\end{Lem}

\begin{proof}
The connectedness can be obtained easily by Lemma \ref{th2.1}. Let $C=[0,1]\times\{0,1\}\cup \{0,1\}\times[0,1]$, then $C$ is a circle, and  $\frac{1}{3}C\subset F_7, \frac{1}{3}(\frac{1}{3}C+{\mathcal{D}})\subset F_7$ (see Figure \ref{fig-lem3.6}). By induction, we can get for any $k\geq 1$,  $$\frac{C}{3^k}+\frac{{\mathcal{D}}_{k-1}}{3^{k-1}}=\frac{C}{3^k}+\frac{{\mathcal{D}}}{3^{k-1}}+\cdots+\frac{{\mathcal{D}}}{3}\subset F_7.$$
Hence
$$\overline{\bigcup_{k=1}^{\infty}({\frac{C}{3^k}+\frac{{\mathcal{D}}_{k-1}}{{3^{k-1}}}})}\subset F_7.$$

On the other hand, for any $x\in F_7$, there exists a sequence $\{d_{j_i}\}_i$ with $d_{j_i}\in {\mathcal{D}}$ such that $x=\sum_{i=1}^{\infty}3^{-i}d_{j_i}.$ By (\ref{digitset}),  for all $k\geq 1$, we have
$\sum_{i=1}^k3^{-i}d_{j_i}\in  3^{-k}{\mathcal{D}}_k\subset\bigcup_{k=1}^{\infty}(\frac{C}{3^k}+\frac{{\mathcal{D}}_{k-1}}{{3^{k-1}}}),$
then $x\in \overline{\bigcup_{k=1}^{\infty}(\frac{C}{3^k}+\frac{{\mathcal{D}}_{k-1}}{{3^{k-1}}})}.$ Hence
$$F_7\subset \overline{\bigcup_{k=1}^{\infty}(\frac{C}{3^k}+\frac{{\mathcal{D}}_{k-1}}{{3^{k-1}}})}.$$
We omit the proof for $F_8$ as it is the same as above.
\end{proof}

\bigskip

A nonempty compact set $T \subset {\mathbb{R}}^2$ is called a {\it tree-like set} if for any two distinct points $x,y\in T,$ there is a unique path (or curve)  in $T$ connecting them.

\begin{figure}[h]
 \centering
 \subfigure[]{
 \includegraphics[width=4cm]{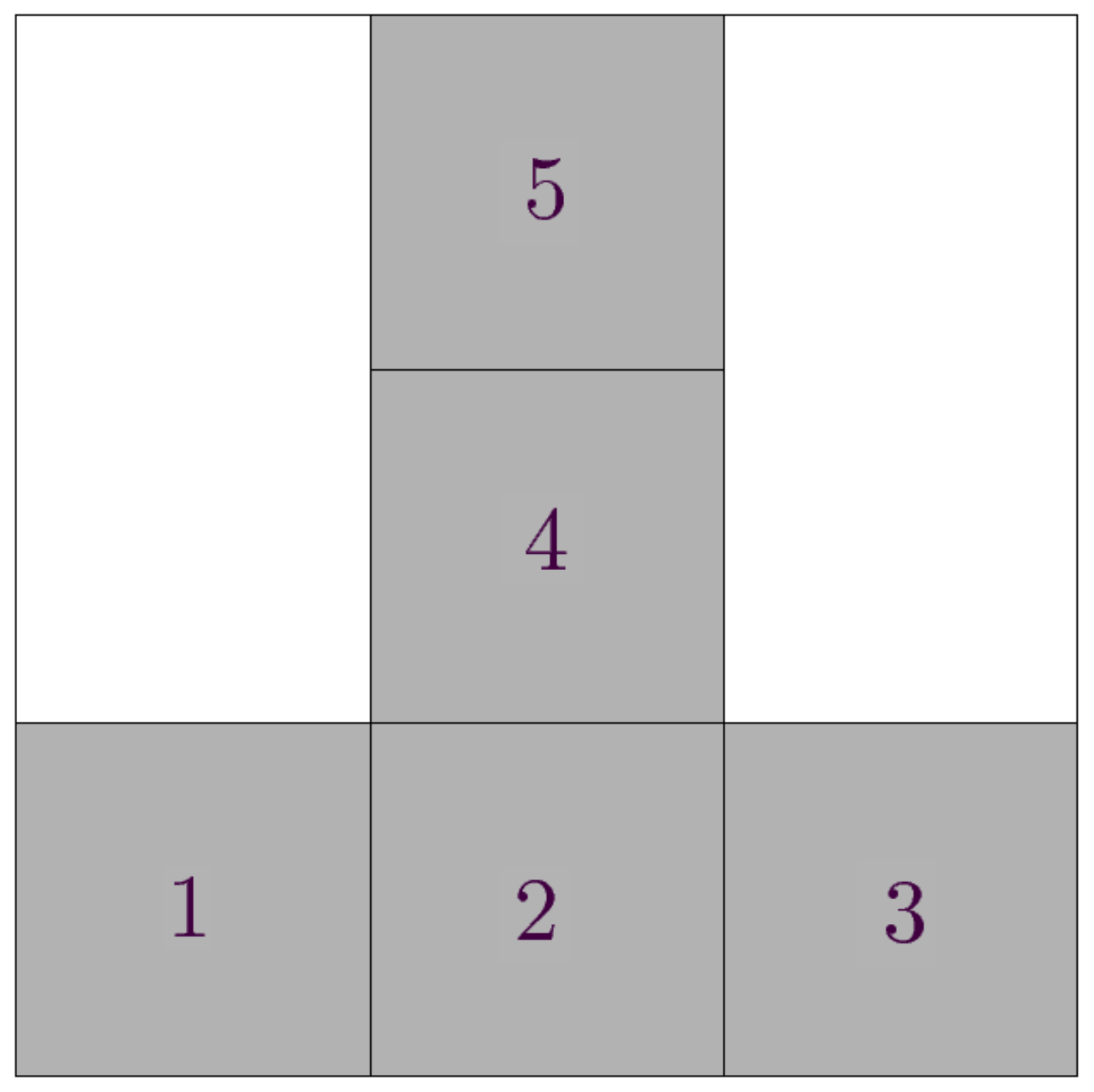}
} \qquad
\subfigure[]{
 \includegraphics[width=4cm]{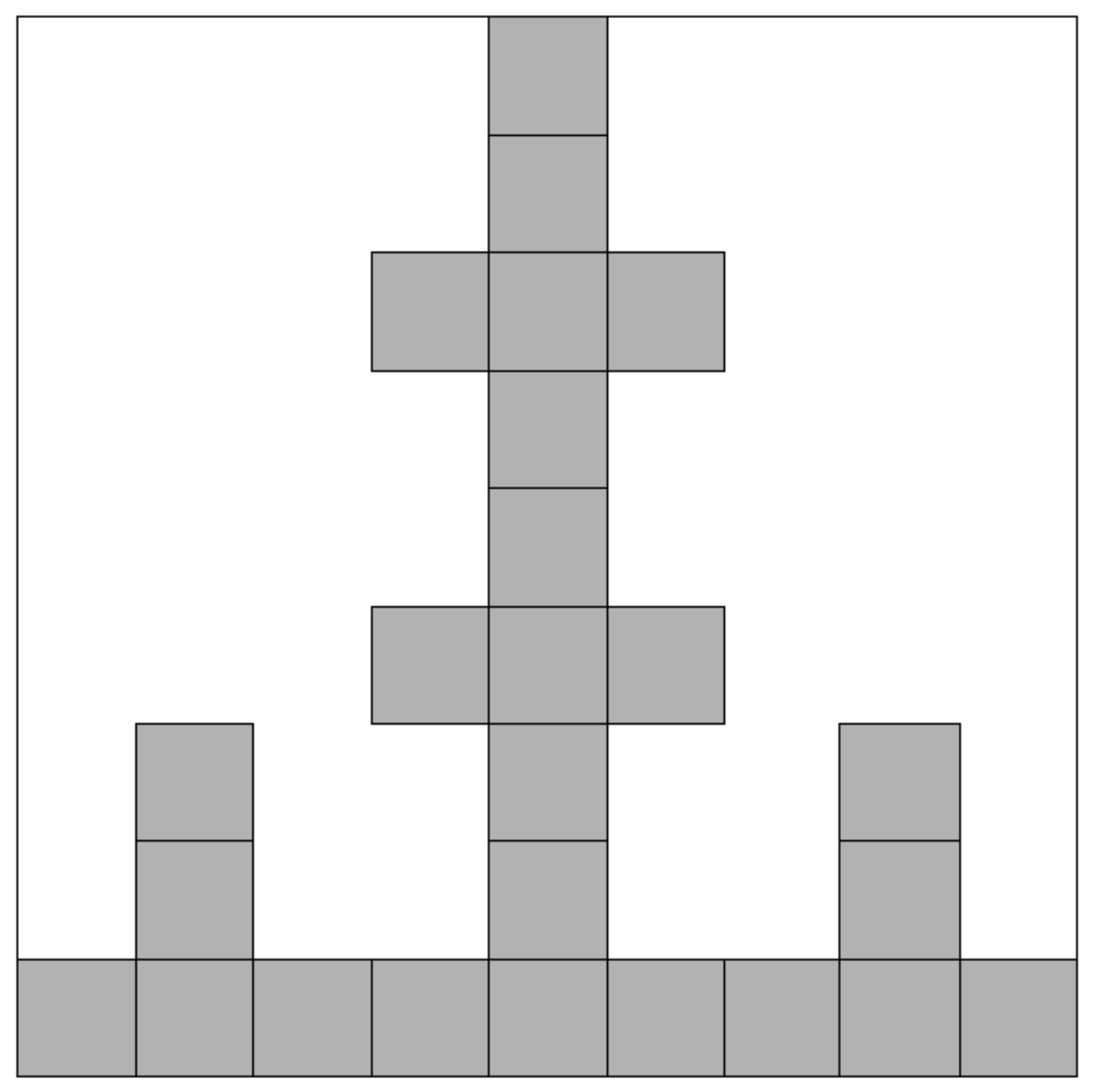}
} \qquad
\subfigure[]{
 \includegraphics[width=4cm]{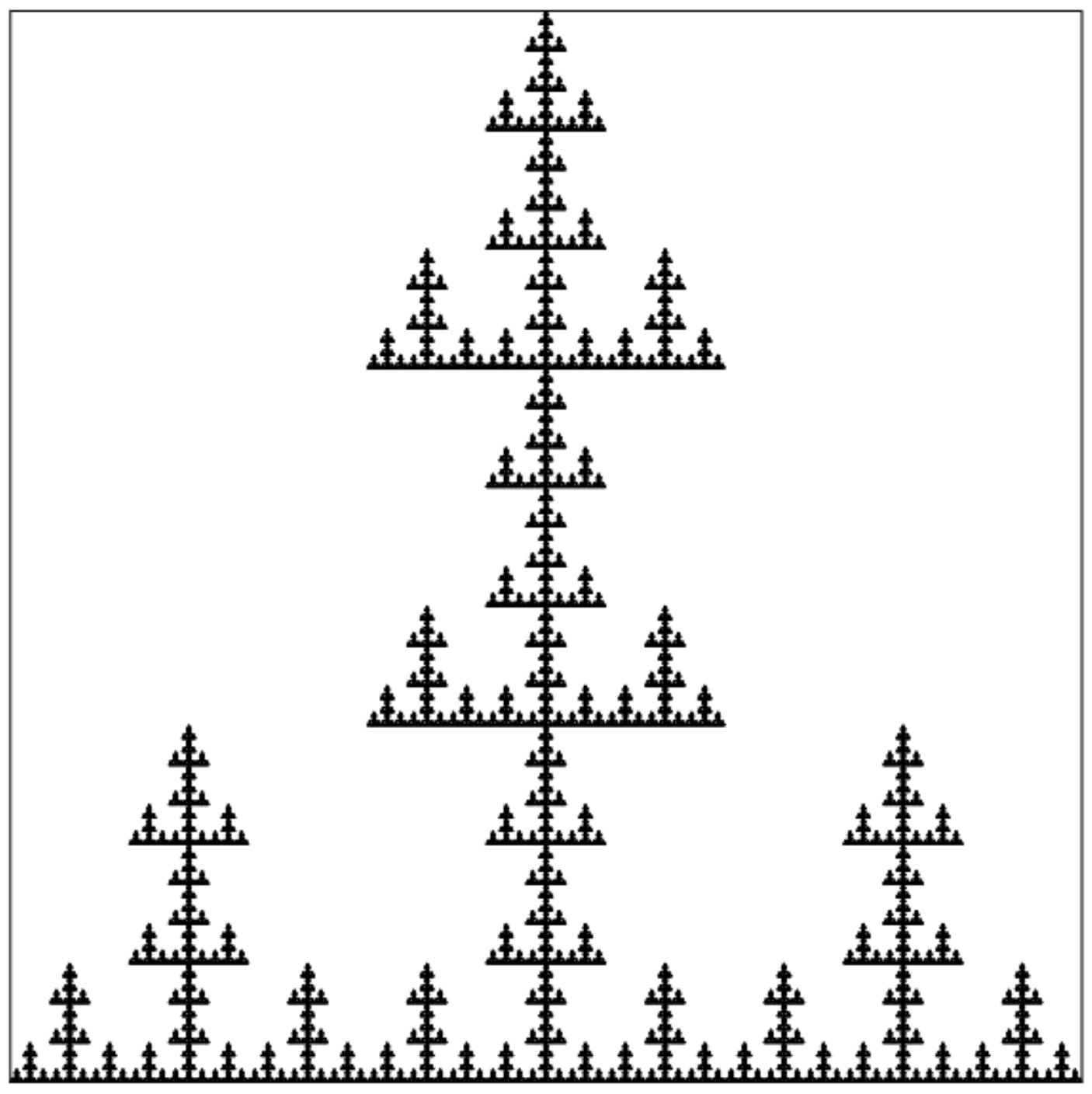}
}
\caption{Fractal square $F_6$}\label{fig.tria}
\end{figure}

\begin{theorem}\label{th3.7}
$F_6$ and $F_7$ are not homeomorphic, hence are not Lipschitz equivalent.
\end{theorem}

\begin{proof}
Lemma \ref{countable square lemma} implies $F_7$ is not a tree-like set, so it suffices to show that $F_6$ is a tree-like set. By Lemma \ref{th2.1}, $F_6$ is connected, hence is pathwise connected \cite{Ha85}. Thus for any two distinct points $x,y\in F_6$, there is a path $\pi(x,y)$ in $F_6$ connecting them. Next we show $F_6$ is a tree-like set by proving the uniqueness of  the path $\pi(x,y)$.

Following notation of (\ref{eq2.1}), let ${\bf u, v}\in \Sigma^k$.  It is known that if $S_{\bf u}(B)\cap S_{\bf v}(B)$ is singleton then $S_{{\bf u}i}(B)\cap S_{{\bf v}j}(B)=\emptyset$ for any $i,j\in \Sigma$; if $S_{\bf u}(B)\cap S_{\bf v}(B)$  is a line segment, say $L_k$, then there exists a unique pair $(i,j)\in{\Sigma}\times{\Sigma}$ such that $S_{{\bf u}i}(B)\cap S_{{\bf v}j}(B)\ ( :=L_{k+1})$ is also a line segment with length $$|L_{k+1}|=\frac{|L_k|}{3}=\frac{1}{3^{k+1}}$$
(see Figure \ref{fig.tria}).  Define
$${\mathcal{E}}_k=\{(d_{\bf u},d_{\bf v}): |S_{\bf u}(B)\cap S_{\bf v}(B)|=\frac{1}{3^k}, d_{\bf u},d_{\bf v}\in{\mathcal{D}}_k\}$$ to be the set of edges for ${\mathcal{D}}_k$.  Then $({\mathcal{D}}_k, {\mathcal{E}}_k)$ forms a tree by the argument above for any $k\ge 1$.

Assume $\pi'(x,y)$ is a path different from  $\pi(x,y)$. Then there exists a point $z_0\in \pi'(x,y)\backslash \{x,y\}$ such that
$${\epsilon}_0:=\inf\{|z-z_0|: z\in\pi(x,y)\backslash \{x,y\}\}> 0.$$
Since $x,y\in F_6 \subset\bigcup_{{\bf u}\in {\Sigma}^k}S_{\bf u}(B)$ and $x\ne y$, there is a large enough $k_0\geq \log_3{\frac{\sqrt{2}}{{\epsilon}_0}}+1$ such that, for any $k\geq k_0$, there exist ${\bf u, v}\in{\Sigma}^k$ such that $x\in S_{\bf u}(B), y\in S_{\bf v}(B)$ and $S_{\bf u}(B)\cap S_{\bf v}(B)=\emptyset$. By the tree structure of $({\mathcal{D}}_k,{\mathcal{E}}_k)$,  we can find a unique finite sequence ${\bf u}_1,\dots,{\bf u}_\ell$ satisfying ${\bf u}={\bf u}_1, {\bf v}={\bf u}_\ell$ and
$({\bf u}_i, {\bf u}_{i+1})\in{{\mathcal{E}}_k}$ for $i=1,\dots,\ell-1.$ Thus $\pi(x,y),\pi'(x,y)\subset \bigcup_{i=1}^\ell S_{{\bf u}_i}(B).$ Suppose $z_0\in S_{{\bf u}_{i_0}}(B)$,  then let $z_1\in\pi(x,y)\cap S_{{\bf u}_{i_0}}(B)$,  we get $$|z_0-z_1|\leq \text{diam}(S_{{\bf u}_{i_0}}(B))=\frac{\sqrt{2}}{3^k}<{\epsilon}_0.$$ That contradicts $|z_0-z_1|\geq {\epsilon}_0$.
\end{proof}

\bigskip

Let $E\subset {\mathbb R}^2$ be a nonempty connected set. We say a point $a\in E$ is a {\it $k$-branch point} if $E\setminus\{a\}$ consists of $k$ components. It is known that $k$-branch points are topological invariants.  A $1$-branch point of $E$ is often called a {\it top} of $E$. The following lemma is obvious.

\begin{Lem} {\label{branch point lemma0}}
Suppose that $x$ is a $k$-branch point in $E \subset {\mathbb R}^2$.  Then for any $U\subset {\mathbb R}^2$,  $E\setminus U$ has at least $k$ components provided that $U$ contains $x$ and the diameter $U$ is small enough.
\end{Lem}

\begin{proof}
Let $E_1,E_2,\dots,E_k$ be the  components of $E\setminus \{x\}$. Let $\delta$ be the minimum of the diameters $\text{diam}(E_j), 1\leq j\leq k.$ If $\text{diam}(U)<\delta/2$, then $(E\setminus U)\cap E_j$ is not empty and contributes at least one component to $E\setminus U$.
\end{proof}

Let $F\in {\mathcal F}_{n,m}$ and $\Sigma=\{1,\dots,m\}$.  For any point $x\in F$,  there exists an infinite word $i_1i_2\cdots$ such that
$$\{x\}=\bigcap_{k=1}^{\infty}S_{i_1\cdots i_k}(F)$$ where $i_j\in \Sigma$ and $S_{i_1\cdots i_k}=S_{i_1}\circ\cdots\circ S_{i_k}$.  We call $i_1i_2\cdots$ a {\it coding} of $x$, and $(F)_{i_1\cdots i_k} := S_{i_1\cdots i_k}(F)$  a {\it cylinder} of $F$.

\begin{Lem}{\label{branch point lemma2}}
Let $\{S_i\}_{i=1}^5$ be  the IFS  of $F_6$, which is depicted by Figure \ref{fig.tria}(a). Suppose $x$ belongs to $F_6$. Then

$(i)$ If the coding of $x$ is unique and contains finitely many symbols $2,4$, then $x$ is a $1$-branch point.

$(ii)$ Suppose the coding of $x$ is unique and contains infinitely many symbols $2,4$. If the coding is not eventually $2$, then $x$ is a $2$-branch point; otherwise $x$ is a $3$-branch point.

$(iii)$ If $x$ has more than one coding, then $x$ is either a $2$-branch or a $4$-branch point.
\end{Lem}

\begin{proof}
$(i)$ Clearly if the coding of $x$ does not contain the symbols $2,4$, then $x$ is a $1$-branch point, namely, $x$ is a top of $F_6$ (see Figure \ref{fig.tria}(c)). Indeed, we can show by induction that if we delete the cylinder $(F_6)_{i_1\cdots i_k}$ from $F_6$, then the resulting set is still connected. Hence $x$ is a $1$-branch point by Lemma \ref{branch point lemma0}.

Now suppose that $i_1i_2\cdots$ contains symbols $2,4$, say $i_k$ is the last symbol in $2,4$ and $i_j$ belongs to $\{1,3,5\}$ for all $j>k$. Then $x$ is a top of the cylinder $(F_6)_{i_1\cdots i_k}$. If $x$ is not a top of $F_6$, then $x$ must belong to another cylinder, which means $x$ has more than one coding.

$(ii)$ Suppose $i_1i_2\cdots$ contains infinitely many symbols $2,4$, and it is not eventually $2$. This means $4$ will appear infinitely many times. Suppose $i_k=4$. Let $U=(F_6)_{i_1\cdots i_k}\setminus\{(F_6)_{i_1\cdots i_k5^{\infty}}\}$. Since $(F_6)_{i_1\cdots i_{k-1}}\setminus U$ consists of two components and $U$ does not intersect other cylinders of $F_6$, we conclude that $F_6\setminus U$ has only two components. Therefore $x$ is a $2$-branch point.

Now suppose that $i_1i_2\cdots$ is eventually $2$. Suppose $i_k=2$ for all $k\geq k_0$. Delete $(F_6)_{i_1\cdots i_{k_0}}$ but keep the three intersecting points with other cylinders, the resulting set consists of three components. Hence $x$ is a $3$-branch point.

$(iii)$ Now suppose $x$ has more than one coding. If $x$ has no coding of eventually $2$, then $x$ must be the common top of two cylinders and it is a $2$-branch point. If $x$ has a coding of eventually $2$, say $i_1\cdots i_k2^{\infty}$. If we delete $x$, then $(F_6)_{i_1\cdots i_k}$ is partitioned into three pieces. The other part of $F_6$  either connects to the top of $(F_6)_{i_1\cdots i_k}$ or connects to $x$. Hence $x$ is a $4$-branch point.
\end{proof}

\begin{figure}[h]
 \centering
 \subfigure[]{
 \includegraphics[width=4cm]{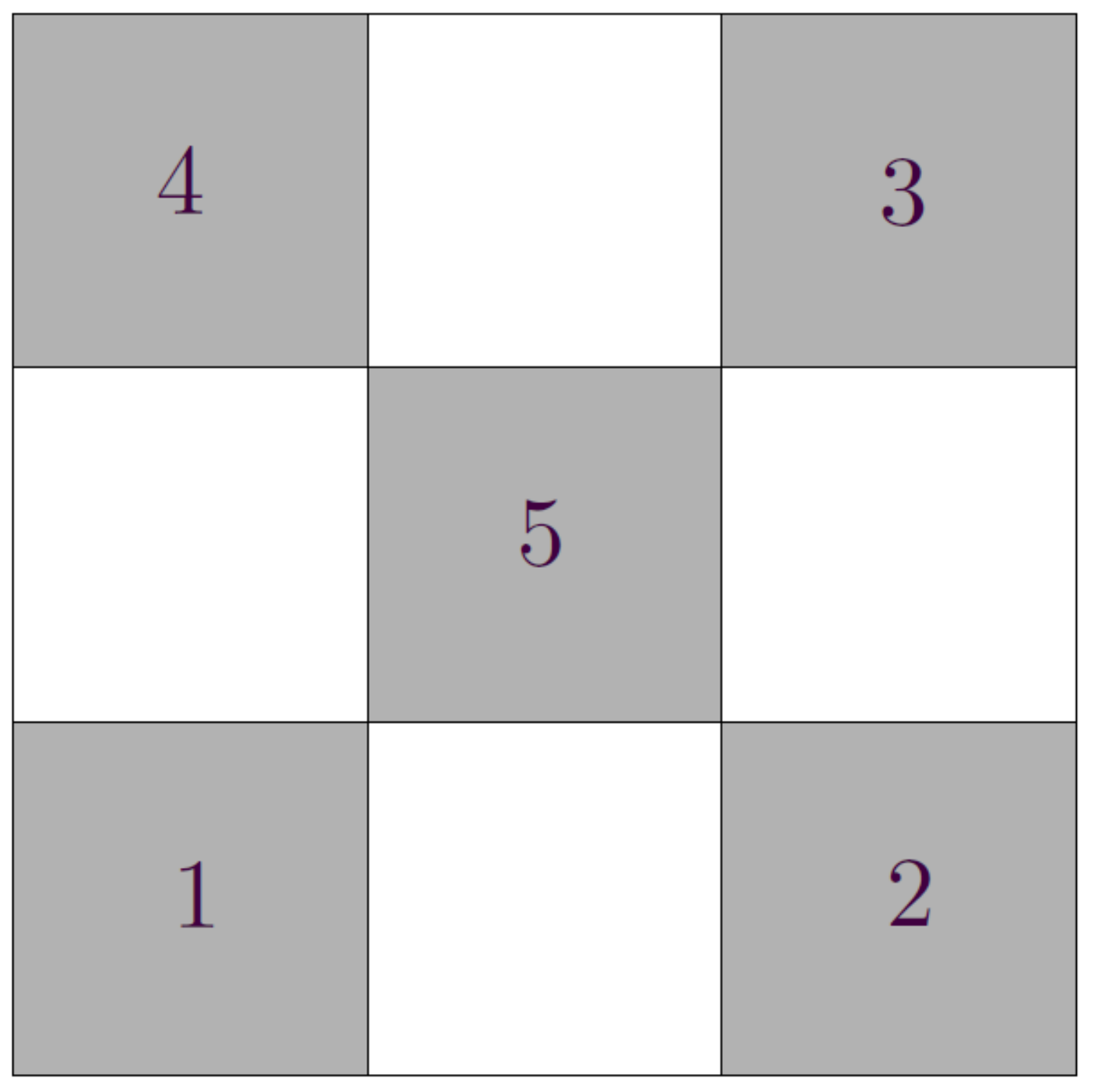}
}\qquad
 \subfigure[]{
 \includegraphics[width=4cm]{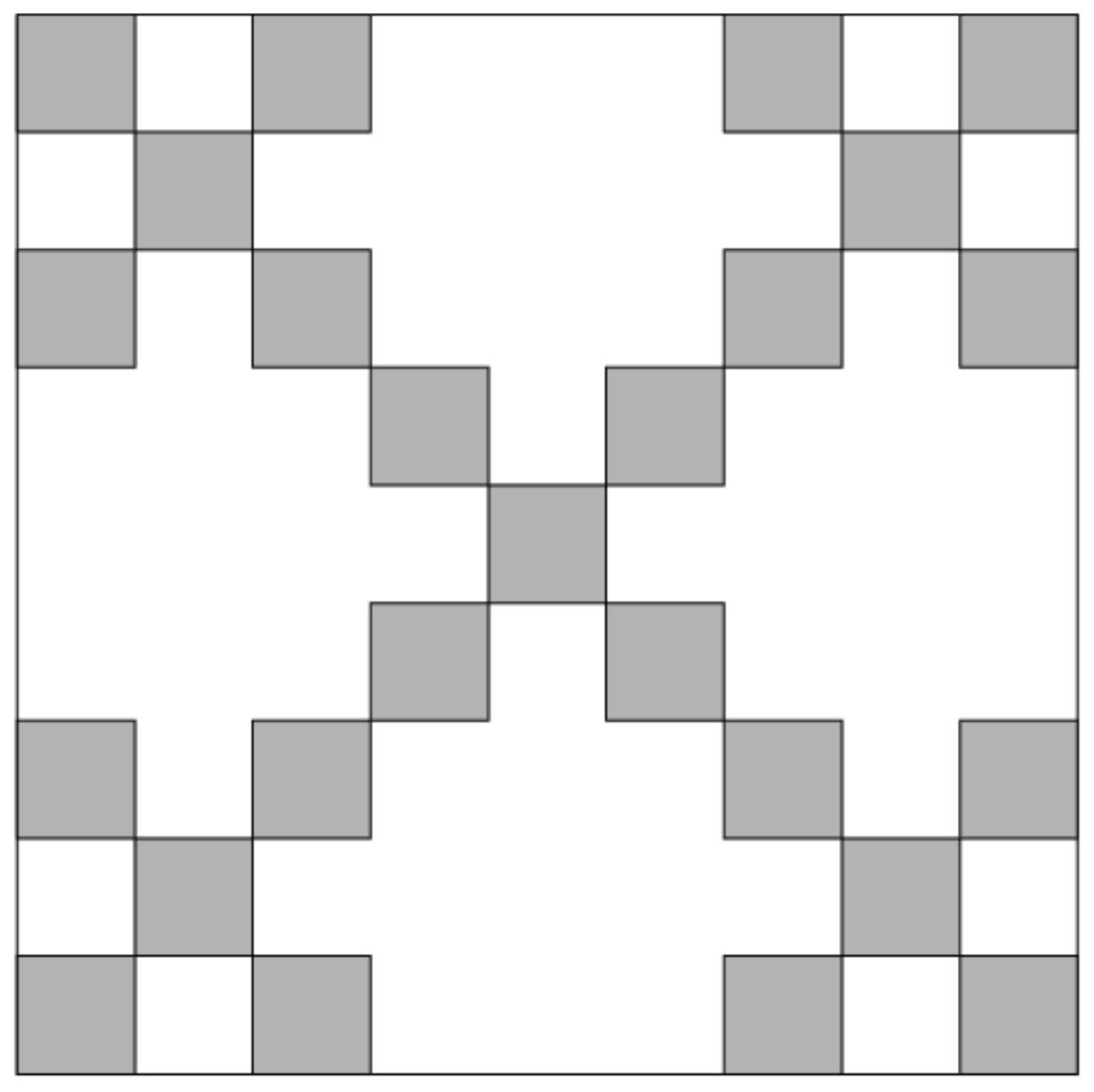}
}\qquad
 \subfigure[]{
 \includegraphics[width=4cm]{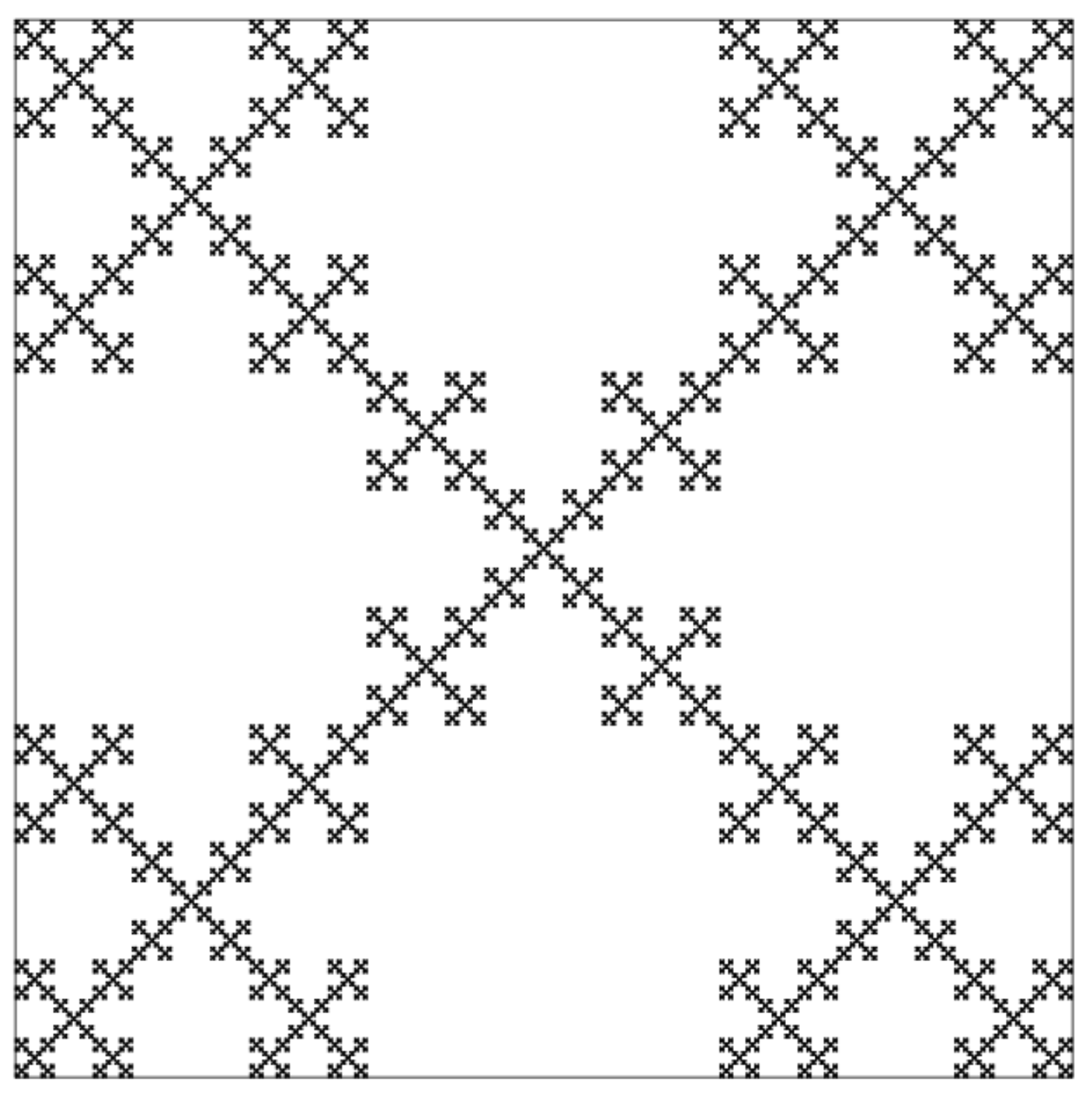}
}
\caption{Fractal square $F_9$}\label{fig.cross}
\end{figure}

\begin{Lem}{\label{branch point lemma1}}
Let $\{S_i\}_{i=1}^5$ be  the IFS  of $F_9$, which is depicted by Figure \ref{fig.cross}(a). Suppose $x$ belongs to $F_9$. Then

$(i)$ If the coding of $x$ is unique and  contains finitely many symbols $5$, then $x$ is a $1$-branch point.

$(ii)$ If the coding of $x$ is unique and contains infinitely many symbols $5$, then $x$ is a $4$-branch point.

$(iii)$ If $x$ has more than one coding, then $x$ is a $2$-branch point.
\end{Lem}

\begin{proof}
$(i)$ Clearly if the coding of $x$ does not contain the symbol $5$, then $x$ is a $1$-branch point, namely, $x$ is a top corner of $F_9$ (see Figure \ref{fig.cross}(c)). Indeed, we can show by induction that if we delete the cylinder $(F_9)_{i_1\cdots i_k}$ from $F_9$, then the resulting set is still connected. Hence $x$ is a $1$-branch point by Lemma \ref{branch point lemma0}.

Now suppose that $i_1i_2\cdots$  contains symbol $5$, say $i_k$ is the last $5$ and $i_j\in\{1,2,3,4\}$ for all $j>k$. Then $x$ is a top of the cylinder  $(F_9)_{i_1\cdots i_k}$. If $x$ is not a top of $F_9$, then $x$ must belong to another cylinder, which means $x$ has more than one coding.

$(ii)$ If $i_1i_2\cdots$ contains infinitely many $5$, suppose $i_k=5$. Let $U=(F_9)_{i_1\cdots i_k}$. Since $(F_9)_{i_1\cdots i_{k-1}}\setminus U$ consists of four components and $U$ does not intersect other cylinders of $F_9$, we conclude that $F_9\setminus U$ has only four components. Therefore $x$ is a $4$-branch point.

$(iii)$ If $x$ has more than one coding, then $x$ must be the common top of two cylinders, and it is a $2$-branch point.
\end{proof}

\begin{theorem}\label{th3.11}
$F_6$ and $F_9$ are not homeomorphic, hence are not Lipschitz equivalent.
\end{theorem}

\begin{proof}
By Lemmas \ref{branch point lemma2}, \ref{branch point lemma1}, we know that $F_6$ contains $3$-branch points, while $F_9$ contains no $3$-branch points. Therefore, they are not homeomorphic.
\end{proof}

\bigskip
\section{Remarks}

Because of  irregularity, it is difficult to  study the remaining $F_{13}, F_{14}$;  $F_{17}, F_{18}$; $F_{20}, F_{21}$ of ${\mathcal F}_{3,5}$. We conjecture that they are not Lipschitz equivalent at all, and  $\#({\mathcal F}_{3,5}/\simeq)=10$.

For the cases of ${\mathcal F}_{3,6}, {\mathcal F}_{3,7}$ and ${\mathcal F}_{3,8}$, we summarize their topological classifications as follows: up to congruence, ${\mathcal F}_{3,6}$ only contains $16$ fractal squares of which $6$  are disconnected and $10$  are connected; ${\mathcal F}_{3,7}$ only contains $8$  connected fractal squares; and ${\mathcal F}_{3,8}$ only  contains $3$ connected fractal squares (please see their figures in the next appendix section). Recently, Ruan and Wang \cite{RuWa} proved that $\#({\mathcal F}_{3,7}/\simeq)=8$ and $\#({\mathcal F}_{3,8}/\simeq)=3$  by making use of an old result called Whyburn's theorem. However, it is still hopeless to handle the other cases completely.

For the general ${\mathcal F}_{n,m}$, we can make a further discussion to get similar results as Section 3, but the process will become more complicated.
\end{section}

\bigskip

\begin{section}{Appendix: Figures of fractal squares}

\begin{figure}[h]
\includegraphics[width=14cm]{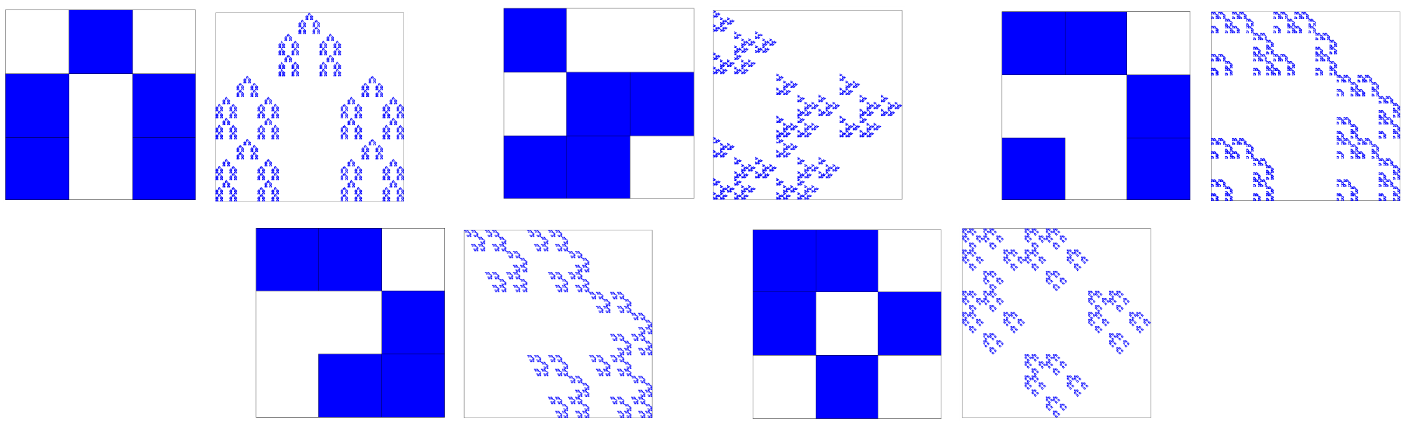}
\caption{Five totally disconnected fractal squares in ${\mathcal F}_{3,5}$}
\end{figure}

\begin{figure} [h]  \label{fig.2}
\includegraphics[width=14cm]{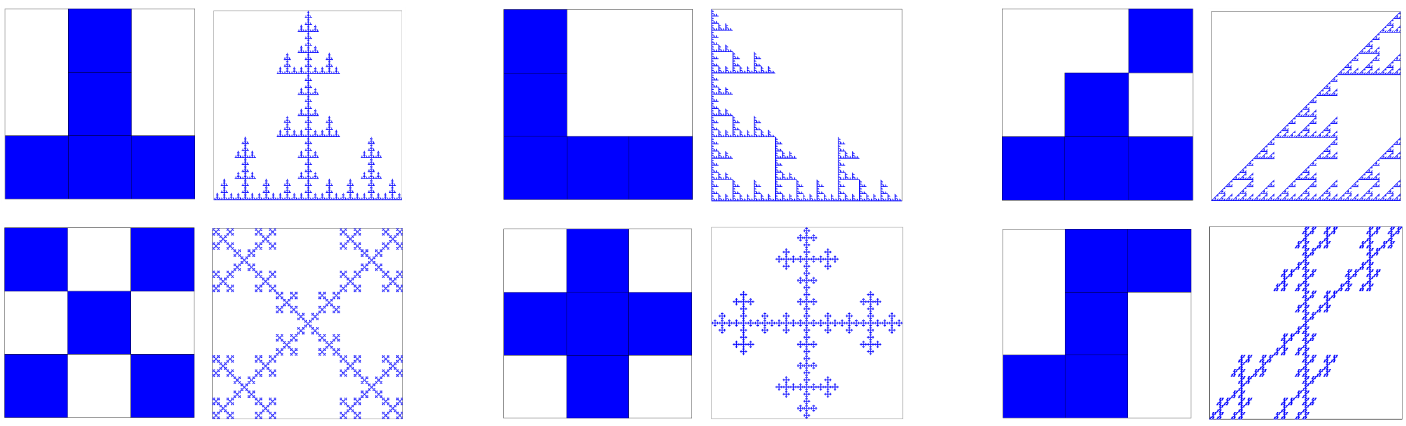}
  \caption{Six connected fractal squares  in ${\mathcal F}_{3,5}$}
\end{figure}

\begin{figure}[h]
\includegraphics[width=14cm]{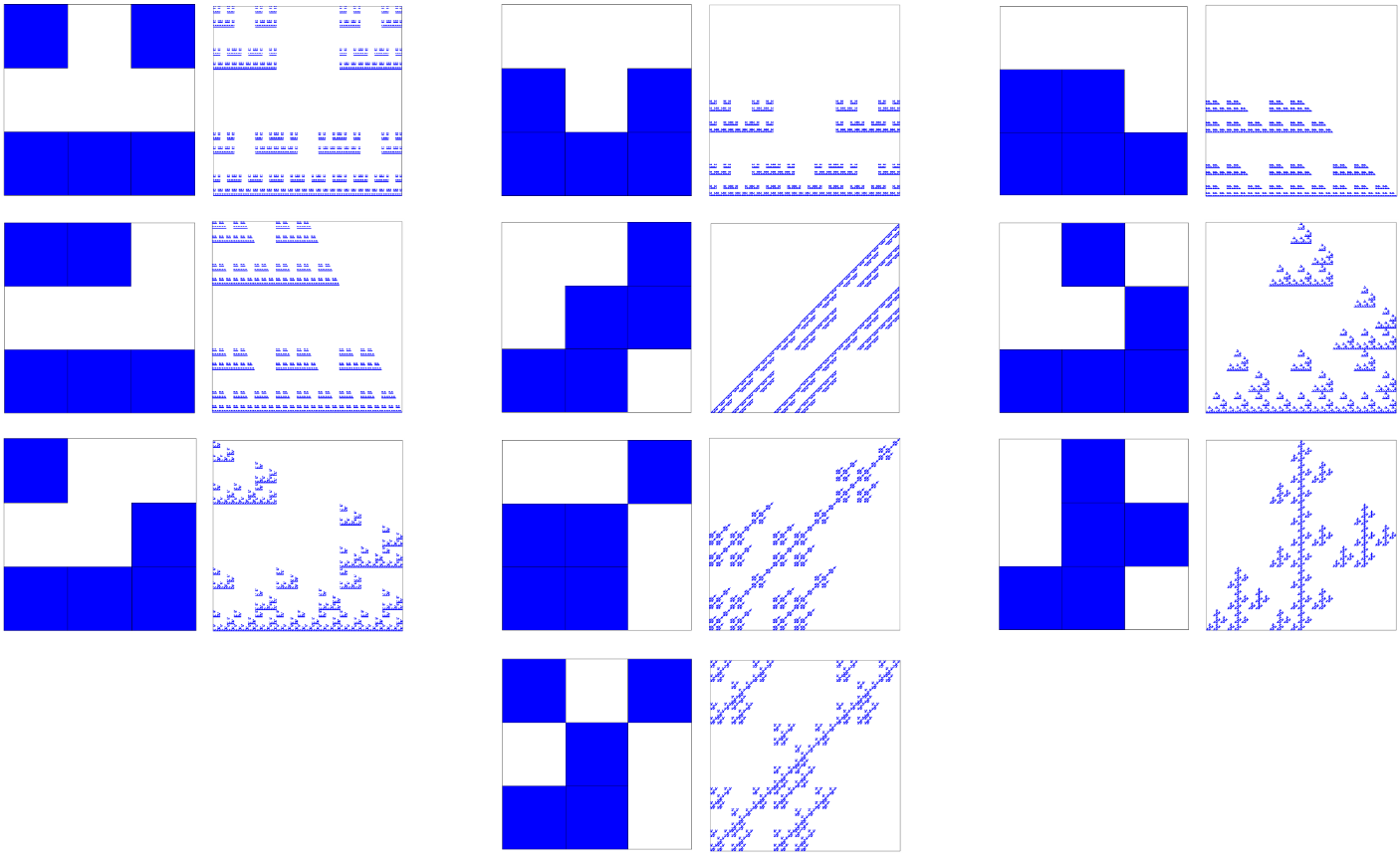}
 \caption{Ten fractal squares containing parallel line segments in ${\mathcal F}_{3,5}$}
\end{figure}

\begin{figure} [h]
\includegraphics[width=14cm]{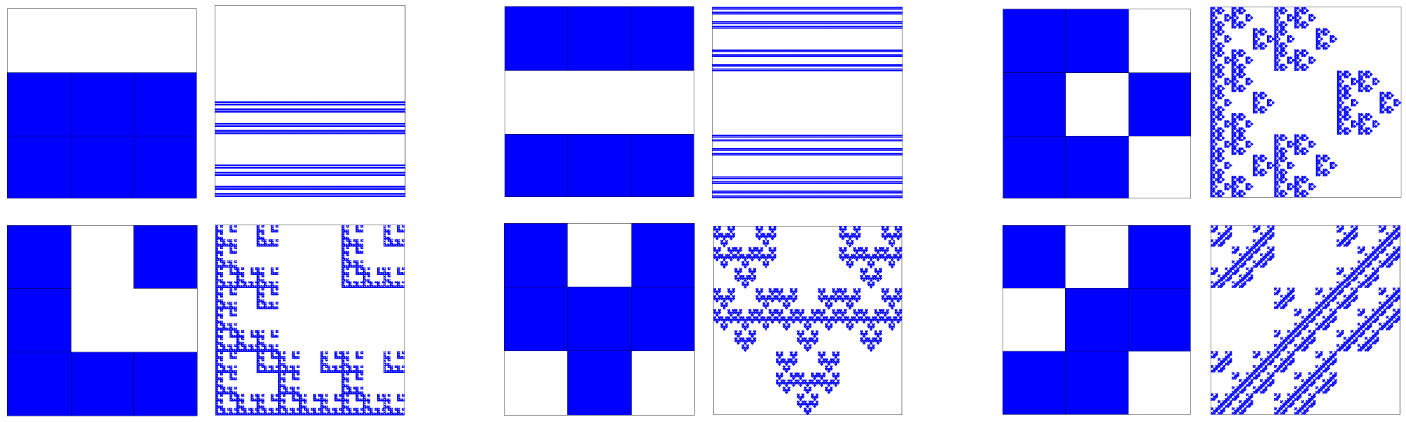}
\caption{Six disconnected fractal squares in ${\mathcal F}_{3,6}$}
\end{figure}

\begin{figure}[h]
\includegraphics[width=14cm]{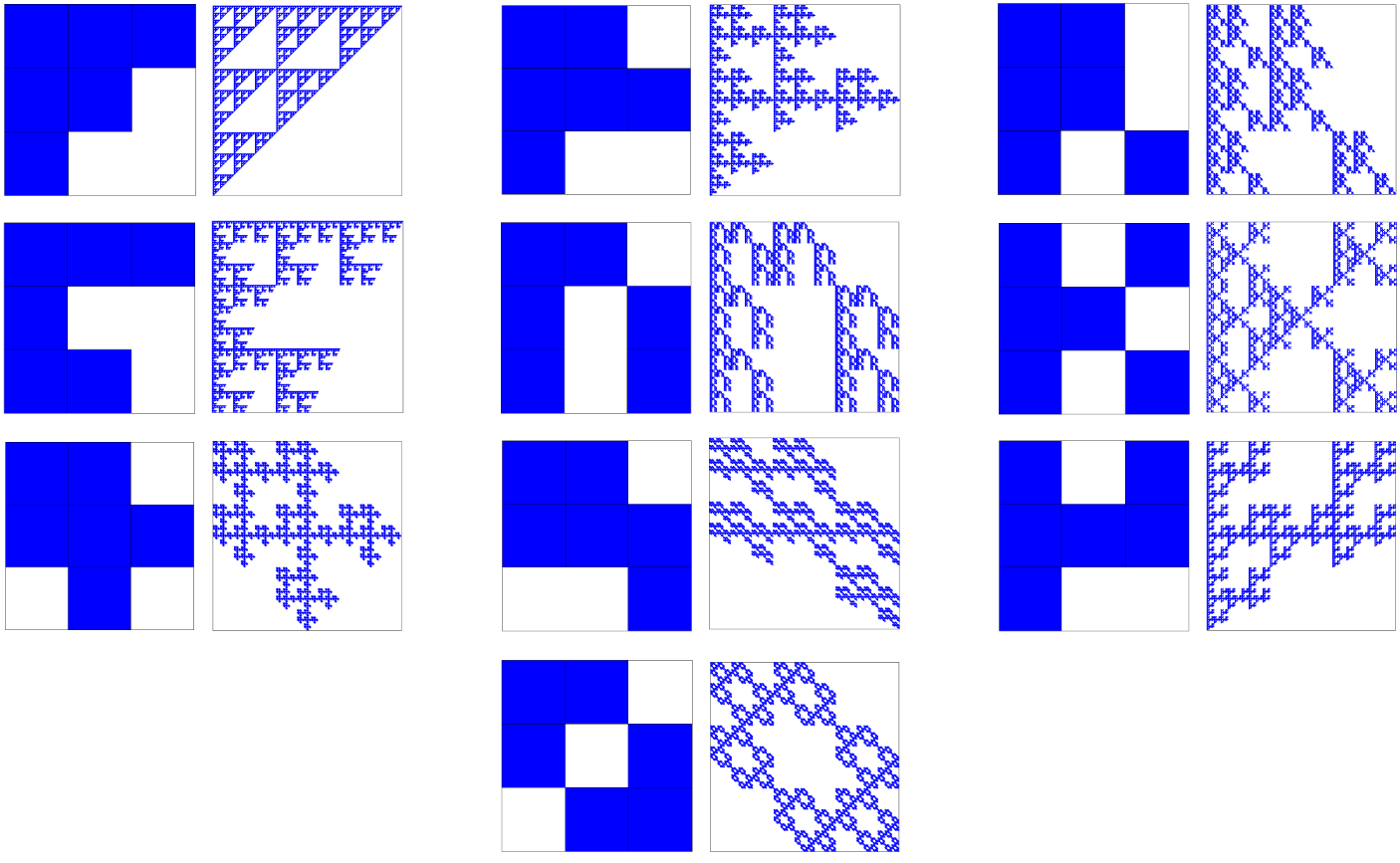}
\caption{Ten connected fractal squares  in ${\mathcal F}_{3,6}$}
\end{figure}

\begin{figure} [h]
\includegraphics[width=14cm]{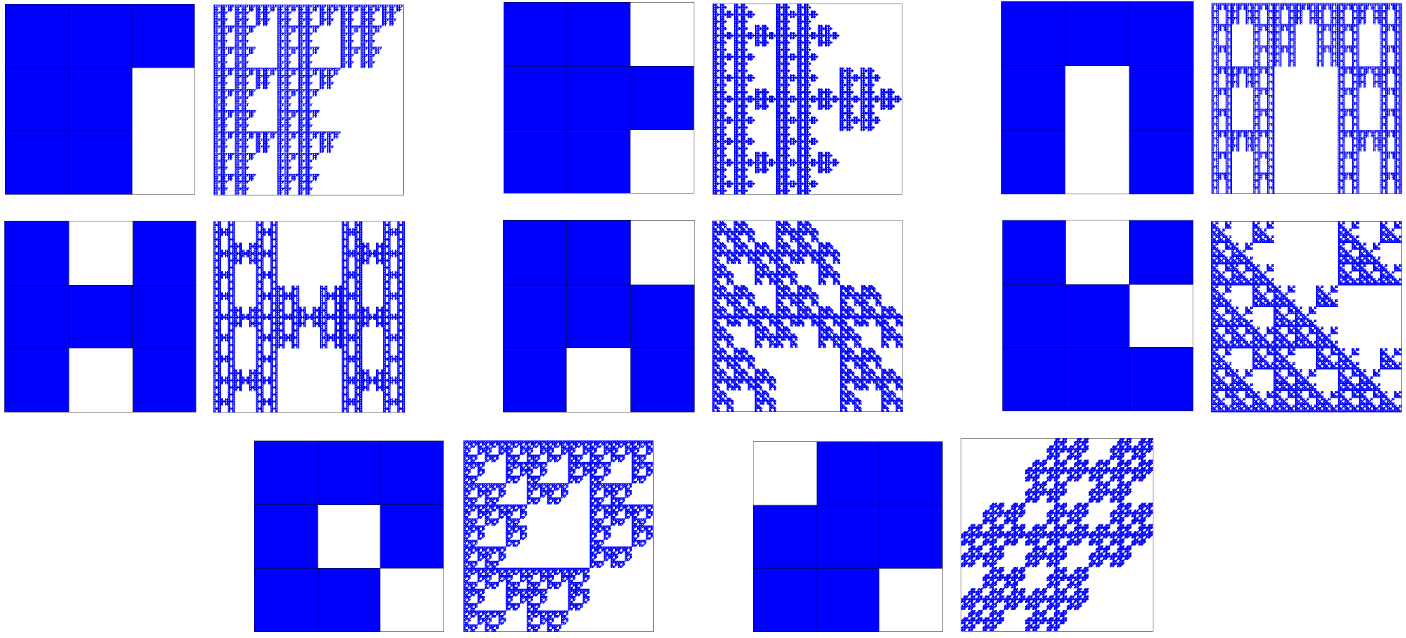}
\caption{Eight fractal squares in ${\mathcal F}_{3,7}$}
\end{figure}

\begin{figure} [h]
\includegraphics[width=14cm]{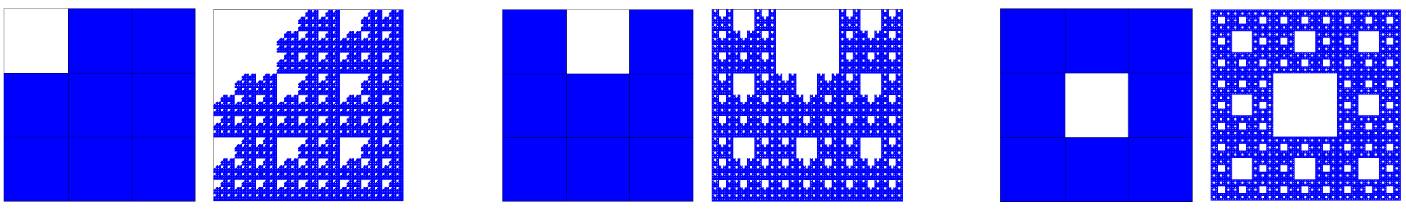}
\caption{Three fractal squares  in ${\mathcal F}_{3,8}$}
\end{figure}
\end{section}

\end{document}